\documentclass[12pt]{amsart}
\usepackage{amsmath,srcltx,amssymb}
\usepackage{fullpage}
\usepackage{stackrel}
\usepackage{cite}
\usepackage{hyperref}

\theoremstyle{plain}
\newtheorem{thm}{Theorem}[section]

\newtheorem{lem}[thm]{Lemma}
\theoremstyle{definition}

\newtheorem{exa}[thm]{Example}
\newtheorem{rem}[thm]{Remark}
\newtheorem{cor}[thm]{Corollary}
\numberwithin{equation}{section}

\newcommand{\R}{\mathbb{R}}
\newcommand{\bR}{{\mathbb R}}
\newcommand{\bN}{{\mathbb N}}
\newcommand{\cK}{{\mathcal K}}
\renewcommand{\a}{\alpha}
\renewcommand{\b}{\beta}
\renewcommand{\c}{\gamma}
\renewcommand{\l}{\lambda}
\newcommand{\e}{\epsilon}

\newcommand{\ol}{\overline}

\renewcommand{\d}{\delta}

\renewcommand{\(}{\left(}
\renewcommand{\)}{\right)}

\newcommand{\lils}[1]{\stackrel[{#1}]{}{\overline\lim}}
\newcommand{\lili}[1]{\stackrel[{#1}]{}{\underline\lim}}

\newcommand{\til}{\tilde}

\DeclareMathOperator{\Id}{Id}
\DeclareMathOperator{\Lip}{Lip}
\DeclareMathOperator{\LB}{LB}
\DeclareMathOperator*{\essinf}{ess\ inf}
\DeclareMathOperator*{\esssup}{ess\ sup}
\allowdisplaybreaks

\begin{document}

\title[Nonlinear perturbed integral equations and nonlocal BVPs]{Nonlinear perturbed integral equations related to nonlocal boundary value problems}
\date{}

\subjclass[2010]{Primary 45G10, secondary  34A34, 34B10, 34B15, 34B18, 34K10}%
\keywords{Perturbed integral equation, nonlocal differential equation, nonlinear boundary condition, nontrivial solution, fixed point index, cone.}%

\author[A. Cabada]{Alberto Cabada}%
\address{Alberto Cabada, Departamento de An\'alise Ma\-te\-m\'a\-ti\-ca, Facultade de Matem\'aticas,
Universidade de Santiago de Com\-pos\-te\-la, 15782 Santiago de Compostela, Spain}%
\email{alberto@cabada.usc.es}%

\author[G. Infante]{Gennaro Infante}
\address{Gennaro Infante, Dipartimento di Matematica e Informatica, Universit\`{a} della
Calabria, 87036 Arcavacata di Rende, Cosenza, Italy}%
\email{gennaro.infante@unical.it}%

\author[F. A. F. Tojo]{F. Adri\'an F. Tojo}
\address{F. Adri\'an F. Tojo, Departamento de An\'alise Ma\-te\-m\'a\-ti\-ca, Facultade de Matem\'aticas,
Universidade de Santiago de Com\-pos\-te\-la, 15782 Santiago de Compostela, Spain}%
\email{fernandoadrian.fernandez@usc.es}%

\begin{abstract}
By topological arguments, we prove new results on the existence, non-existence, localization and multiplicity of nontrivial solutions of a class of perturbed nonlinear integral equations. These type of integral equations arise, for example, when dealing with boundary value problems where nonlocal terms occur in the differential equation and/or in the boundary conditions. Some examples are given to illustrate the theoretical results.
\end{abstract}

\maketitle

\section{Introduction}
Infante and Webb~\cite{gijwems}, by means of classical fixed point index theory, studied the existence of multiple~\emph{nontrivial} solutions of perturbed integral equations of the type 
\begin{equation}\label{pham-gijw}
u(t)=\c(t)\hat{{\alpha}}[u]+\int_0^1 k(t,s)f(s,u(s))\,ds,\end{equation}
where $\gamma$ is a continuous function, allowed to change sign, and $\hat{{\alpha}}[\cdot]$ is an~\emph{affine} functional given by a Stieltjes integral with a \emph{positive} measure, namely
\begin{equation*}
\hat{{\alpha}}[u]=A_0+\int_0^1 u(s)\,dA (s).
\end{equation*}
Equation~\eqref{pham-gijw} can be used to study some nonlocal boundary value problems (BVPs) occurring when modelling the steady-state of a heated bar of length one subject to a thermostat, where a controller in one end adds or removes heat accordingly to the temperature measured by sensor at a point of the bar. This type of heat-flow problem was motivated by earlier work by Guidotti and Merino~\cite{guimer} and has been investigated by a number of authors --we refer the reader to the recent papers \cite{gi-pp-ft, jw-narwa} and references therein.

The approach of \cite{gijwems} was modified by Cabada and co-authors \cite{ac-gi-at-tmna} in order to deal with the case of integral equations with a deviated argument, namely
\begin{equation}\label{pham-cit}
u(t)=\c(t){\alpha}[u]
+\int_{0}^{1}k(t,s)g(s)f(s,u(s),u(\sigma(s)))ds,
\end{equation}
where
$\sigma$ is a continuous function such that $\sigma([0,1])\subseteq [0,1]$ and $\alpha[\cdot]$ is a \emph{linear} functional, given by a Stieltjes integral, that is
\begin{equation}\label{eqst}
\alpha [u]=\int_0^1 u(s)\,dA (s),
\end{equation}
with a \emph{signed} measure, in the spirit of the paper by Webb and Infante \cite{jwgi-nodea-08}.
The results of \cite{ac-gi-at-tmna} cover the interesting case of differential equations with reflections and, in particular were applied to the study of the BVP
\begin{equation}\label{dde}
u''(t)+g(t)f(t,u(t),u(\sigma(t)))=0,\ t \in (0,1),
\end{equation}
\begin{equation}\label{ddebcs}
u'(0)+{\alpha}[u]=0,\; \beta u'(1) + u(\eta)=0,\; {\eta}\in [0,1].
\end{equation}
The BVP \eqref{dde}-\eqref{ddebcs}
arises when studying the steady states of a model of a light bulb with a temperature regulating system that includes a feedback controller. One assumption made in this thermostat model is that the feedback controller has a \emph{linear} response; for more details see Section 4~of \cite{ac-gi-at-tmna}.

The formulation of nonlocal boundary conditions (BCs) in terms of Stieltjes integrals is fairly general and includes, as special cases, multi-point and integral conditions, namely
\begin{displaymath}\a[u]=\sum_{j=1}^m\a_j u(\eta_j)\quad\text{or}\quad \a[u]=\int_{0}^1\phi(s)u(s)ds.\end{displaymath}
The study of multi-point problems has been initiated, as far as we know, in 1908 by Picone~\cite{Picone}. For an introduction to nonlocal problems we refer to the reviews of Whyburn~\cite{Whyburn}, Conti~\cite{Conti}, Ma~\cite{rma}, Ntouyas~\cite{sotiris} and \v{S}tikonas~\cite{Stik} and to the papers by Karakostas and Tsamatos~\cite{kttmna, ktejde} and Webb and Infante~\cite{jwgi-lms, jw-gi-jlmsII}.

Webb and Infante~\cite{jw-gi-jlmsII} gave a unified method for establishing the existence of \emph{positive} solutions of a large class of ordinary differential equations of arbitrary order, subject to nonlocal BCs. The methodology in~\cite{jw-gi-jlmsII} involves the fixed point index and, in particular deals with the integral equation
\begin{equation}\label{pham-jwgi}
u(t)=\sum_{i=1}^{N}\gamma_{i}(t) \alpha_{i}[u]+\int_{0}^{1}
k(t,s)g(s) f(s,u(s))\,ds.
\end{equation}
Here the functions $\gamma_{i}$ are nonnegative and the \emph{linear} functionals $\alpha_{i}[\cdot]$ are of the type \eqref{eqst}. The results of~\cite{jw-gi-jlmsII} are well suited for dealing with differential equations of arbitrary order with many nonlocal terms. These results were applied to the study of fourth order problems that model the deflection of an elastic beam.

A common feature of the integral equations~\eqref{pham-gijw}, \eqref{pham-cit} and \eqref{pham-jwgi} is the fact that these equations are designed to deal with BVPs where the boundary conditions involve at most \emph{affine} functionals. In physical models this corresponds to feedback controllers having a \emph{linear} response. Nevertheless, in a number of applications, the response of the feedback controller can be \emph{nonlinear}; for example the nonlocal BVP 
\begin{equation}\label{cbvp}
u^{(4)}(t)-g(t)f(t,u(t))=0,\ u(0)=u'(0)=u''(1)=0,\; u'''(1)+\hat{B}(u(\eta))=0,
\end{equation}
describes a cantilever equation with
a feedback mechanism, where a spring reacts (in a nonlinear manner) to the displacement registered in a point $\eta$ of the beam. Positive solutions of the BVP~\eqref{cbvp} were investigated by Infante and Pietramala in \cite{gipp-cant} by means of the perturbed integral equation
\begin{equation*}
u(t)=\gamma(t)\hat{B}(\hat{\alpha}[u])+ \int_{0}^{1} k(t,s)g(s)f(s,u(s))\,ds,
\end{equation*}
where $\hat{B}:\mathbb{R}^{+}\to \mathbb{R}^{+}$ is a continuous, possibly nonlinear function.

Note that the idea of using perturbed Hammerstein integral equations in order to deal with the existence of solutions of BVPs with nonlinear 
BCs has been used with success in a number of papers, see, for example, the manuscripts of Alves and co-authors~\cite{amp}, Cabada~\cite{Cabada1}, Franco et al.~\cite{dfdorjp}, Goodrich~\cite{Goodrich3, Goodrich4, Goodrich5, Goodrich6, Goodrich7}, Infante~\cite{gi-caa}, Karakostas~\cite{kejde}, Pietramala~\cite{paola}, Yang~\cite{ya1,ya2} and references therein.

The existence of  \emph{nontrivial} solutions of the BVP
\begin{equation*}
u''(t)+g(t)f(u(t))=0, u'(0)+\hat{B}(\hat{{\alpha}}[u])=0,\; \beta u'(1) + u(\eta)=0,
\end{equation*}
that models a heat-flow problem with a nonlinear controller, were discussed by Infante~\cite{gi-atl}, by means of the perturbed integral equation
\begin{equation*}
u(t)=\gamma(t)\hat{B}(\hat{{\alpha}}[u])+\int_{0}^{1}k(t,s)g(s)f(u(s))\,ds.
\end{equation*}

On the other hand, BVPs where nonlocal terms occur in the differential equation have been studied by a number of authors. For example, the case of equations with reflection of the argument has been investigated by
Andrade and Ma~\cite{And}, Cabada and co-authors~\cite{ac-gi-at-bvp}, Piao~\cite{Pia1, Pia2}, Piao and Xin~\cite{Pia3},  
Wiener and Aftabizadeh~\cite{Wie1},
the case of equations with deviated arguments has be
en studied by Jankowski~\cite{Jan1, Jan3, Jan6}, Figueroa and Pouso~\cite{rub-rod-lms} and Szatanik~\cite{Sza1, Sza2} and the case of equations that involve the average of the solution has been considered by Andrade and Ma~\cite{And}, Chipot and Rodrigues~\cite{Chipot} and Infante~\cite{gi-caa2}.
 
Here we continue the study of~\cite{ac-gi-at-tmna, gi-atl, gi-caa2}  and discuss the existence of multiple nontrivial solutions of perturbed Hammerstein integral equations of the kind
\begin{equation*}
u(t)=Bu(t)+\int_0^1 k(t,s)g(s)f(s,u(s),Du(s))\,ds,
\end{equation*}
where
$B:C(I)\to C(I)$ is a compact and continuous map, $D:C(I)\to L^\infty(I)$ a continuous map and $f$ is a non-negative $L^\infty$-Carath\'eodory function. In our setting $B$ and $D$ are possibly nonlinear.
This type of integral equation arises naturally when dealing with a BVP where nonlocal terms occur in the differential equation and in the BCs. Here we prove the existence of multiple solutions that are allowed to \emph{change sign}, in the spirit of the earlier works~\cite{gijwjiea, gijwjmaa, gijwems}. 
The methodology relies on the use of the theory of fixed point index. Some of our criteria involve the principal eigenvalue of an associated linear operator. We make use of ideas from the papers~\cite{ac-gi-at-tmna, gijwjiea, gi-pp-ft, jw-tmna, jwgi-lms, jwkleig} and our results complement the ones of \cite{ac-gi-at-tmna, gi-caa, gi-pp-ft, jw-gi-jlmsII}.

In the last Section, for illustrative purposes we study, in two examples, the nonlocal differential equation
\begin{equation*}
u''(t)+f(t,u(t))+\c(t)u(\eta(t))=0,
\end{equation*}
subject to different BCs, showing that the constants occurring in our theoretical results can be computed.

\section{The integral operator}
Let $I:=[0,1]$, $\bR^+=(0,+\infty)$. We work in the space $C(I)$ of the continuous functions on $ I $ endowed with the usual norm
 $\|w\|:=\max_{t\in I} |w(t)|$. We also use the space $ L^{\infty}(I) $, where we denote (with an abuse of notation) its norm by $\|w\|:= \esssup_{t\in I}  |w(t)|$.
In this section we obtain results for the fixed points of the integral operator
\begin{equation}\label{eqthamm}
Tu(t)=Bu(t)+\int_0^1 k(t,s)g(s)f(s,u(s),Du(s))\,ds,
\end{equation}
where
$B:C(I)\to C(I)$ is a continuous and compact map, $D:C(I)\to L^{\infty}(I)$ a continuous map and $f$ is a non-negative $L^\infty$-Carath\'eodory function. $B$ and $D$ are not necessarily linear.
\par
Given $u: I\to\bR$, we define $u^+(s):=\max\{u(s),0\}$, $u^-(s):=\max\{-u(s),0\}$. We recall that a \emph{cone} $K$ in a Banach space $X$  is a closed
convex set such that $\lambda \, x\in K$ for $x \in K$ and
$\lambda\geq 0$ and $K\cap (-K)=\{0\}$. We denote by $P$ the cone of non-negative functions in $C(I)$. 

We make the following assumptions on the terms that occur in \eqref{eqthamm}.
\begin{enumerate}
\item [$(C_{1})$] $k: I \times I \rightarrow \bR$ is measurable, and for every $\tau\in
 I $ we have
\begin{equation*}
\lim_{t \to \tau} |k(t,s)-k(\tau,s)|=0 \;\text{ for almost every } s \in
 I .
\end{equation*}{}
\item [$(C_{2})$]
 There exist a subinterval $[a,b] \subseteq  I $, a function
$\Phi \in L^{1}(I)$, and a constant $c_{1} \in (0,1]$ such
that
\begin{align*}
|k(t,s)|\leq \Phi(s) \text{ for } &t \in  I  \text{ and almost
every } \, s\in  I,\\
 k(t,s) \geq c_{1}\Phi(s) \text{ for } &t\in
[a,b] \text{ and almost every } \, s \in  I .
\end{align*}
\item [$(C_{3})$] $g,\,g\,\Phi \in L^1( I) $, $g(t) \geq 0$ for almost every $t \in I$, and $\int_a^b
\Phi(s)g(s)\,ds >0$.{}
\item  [ $(C_{4})$] There exist measurable functions $f_i:I\times \R \to [0,\infty)$, $\gamma_{ij}:I\to\bR$, $j=1,\dots,m_i$, $\d_{ij}:I\to\bR$, $j=1,\dots,n_i$, continuous functionals $\a_{ij}:C(I)\to\bR$, $j=1,\dots,m_i$ and $\b_{ij}:C(I)\to\bR$, $j=1,\dots,n_i$, $i=1,2$ and a constant $c\in(0,c_1]$ such that the set
$$K:=\{u\in C(I) \ :\ \min_{t\in [a,b]}u(t)\ge c\|u\|,\ \a_{ij}[u],\b_{ij}[u]\ge0\},$$
is a cone satisfying the following inequalities for every $u\in K$:
\begin{align*} & \sum_{j=1}^{m_1}\c_{1j}(t)\a_{1j}[u]+f_1(t,u(t))\le f(t,u(t),Du(t)), \text{ for every } t\in[a,b],\\
& \sum_{j=1}^{n_1}\d_{1j}(t)\b_{1j}[u]\le Bu(t), \text{ for every } t\in[a,b],\\
 & f(t,u(t),Du(t))\le\sum_{j=1}^{m_2}\c_{2j}(t)\a_{2j}[u]+f_2(t,u(t)), \text{ for every } t\in I,\\ &
Bu(t)\le \sum_{j=1}^{n_2}\d_{2j}(t)\b_{2j}[u],\ |Bu(t)|\le \sum_{j=1}^{n_2}|\d_{2j}(t)|\b_{2j}[u], \text{ for every } t\in I.\end{align*}

\item[$(C_{5})$] The nonlinearities $f:I\times\bR^2\to [0,+\infty)$, $f_1:I\times\bR\to [0,+\infty)$ and $f_2:I\times\bR\to [0,+\infty)$ satisfy $L^\infty$-Carath\'{e}odory conditions, that is $f(\cdot,u,v)$, $f_i(\cdot,u)$ are measurable for each fixed $u,v\in \R$; $f(t,\cdot,\cdot)$, $f_i(t,\cdot)$ are continuous for a.\,e. $t\in  I $, and for each $r>0$, there exists $\phi_{r} \in
L^{\infty}(I) $ such that
\begin{equation*}
f(t,u,v),f_i(t,u)\le \phi_{r}(t) \;\text{ for all } \; u,v\in
[-r,r],\;\text{ and a.\,e. } \; t\in  I .
\end{equation*}
\item[$(C_{6})$] $\c_{ij}\in C (I)$. Let $\til\c_{ij}(t):=\int_0^1|k(t,s)|g(s)\c_{ij}(s)ds$. Assume that the families of functions $\{\til\c_{ij},\d_{ij}\}_{i,j}$ belong to $K\backslash\{0\}$.
\item[$(C_{7})$]Define $\varphi_i=(\a_{i1},\dots,\a_{im_i},\b_{i1},\dots,\b_{in_i})$, $\psi_i=(\til\c_{i1},\dots,\til\c_{im_i},\d_{i1},\dots,\d_{in_i})$ and denote by $\varphi_{ij}$ and $\psi_{ij}$ the $j$-th element of $\varphi_i$ and $\psi_i$ respectively.
We have the following inequalities.
\begin{align}\label{superlin}
\varphi_{1j}[\tau_1 u+\tau_2v]\ge & \tau_1\varphi_{1j}[u]+\tau_2\varphi_{1j}[v],\ \tau_1,\tau_2\in\bR^+,\ u,v\in K,\ j=1,\dots, m_1+n_1,\\\label{sublin}
\varphi_{2j}[\tau_1 u+\tau_2v]\le & |\tau_1|\,\varphi_{2j}[u]+|\tau_2|\,\varphi_{2j}[v],\ \tau_1,\tau_2\in\bR,\ u,v\in K,\ j=1,\dots, m_2+n_2,\\ \label{e-phi-lambda}
\varphi_{2j}[\tau u]\ge & \varphi_{2j}[u],\ \tau \ge 1, \ u,v\in K,\ j=1,\dots, m_2+n_2.
\end{align}

Furthermore, assume that 
\begin{align*}
 \mathcal{K}_{\varphi_{1j}}(s) & :=\varphi_{1j}[k(\cdot,s)]\ge 0,\text{ for a.\,e.}\ s\in I,\ \mathcal{K}_{\varphi_{1j}}\in L^\infty(I),\ j=1,\dots,m_1+n_1, \\
 \mathcal{K}_{\varphi_{2j}}(s) & :=\varphi_{2j}[|k(\cdot,s)|]\ge 0,\text{ for a.\,e.}\ s\in I,\ \mathcal{K}_{\varphi_{2j}}\in L^\infty(I),\ j=1,\dots,m_2+n_2,
\end{align*}
\begin{equation}
 \label{isuperlin}
 \begin{aligned}
\varphi_{1j}\left[\int_a^bk(\cdot,s)g(s)f_1(s,u(s))ds\right]\ge  \int_a^b\varphi_{1j}[k(\cdot,s)]g(s)f_1(s,u(s))ds, & \\ u\in K,\ j=1,\dots, m_1+n_1, &
\end{aligned}
\end{equation}
\begin{equation}
\begin{aligned}
\label{isublin}
\varphi_{2j}\left[\int_0^1|k(\cdot,s)|g(s)f_2(s,u(s))ds\right]\le  \int_0^1\varphi_{2j}[|k(\cdot,s)|]g(s)f_2(s,u(s))ds, & \\ u\in K,\ j=1,\dots, m_2+n_2. &
\end{aligned}
\end{equation}

\item [$(C_{8})$]  Define $M_k=(\varphi_{ki}[\psi_{kj}])_{i,j=1}^{m_k+n_k}\in\mathcal{M}_{m_k+n_k}(\bR),\ k=1,2$. Assume that their respective spectral radii $r$ satisfy that $r(M_1)<1/c_1$ and  $r(M_2)<1$.

\item [$(C_{9})$] Let $c$ and $K$ be given in $(C_4)$ and assume that
\begin{displaymath}\sum_{j=1}^{n_1}\d_{1j}(t)\b_{1j}[u]\ge c\sum_{j=1}^{n_2}\|\d_{2j}\|\b_{2j}[u]\text{ for every } t\in[a,b] \; \mbox{and}\; u\in K.\end{displaymath}

\item [$(C_{10})$] $\varphi_{1j}[u]\ge\varphi_{1j}[v]$ for every $u,v\in K$ such that $u(t)\ge v(t)$ for all $t\in [a,b]$, $\varphi_{2j}[u]\ge\varphi_{2j}[v]$ for every $u,v\in K$ such that $u(t)\ge v(t)$ for all $t\in I$ and $\varphi_{ij}[u]\ge0$ for every $u\in P$.\\
We also assume $\varphi_{ij}[Tu],\varphi_{ij}[F_1u],\varphi_{ij}[L_1u]\ge0$ for every $u\in K$ where, for $t\in[0,1]$,
\begin{align*}F_1u(t):= & \int_{a}^{b}k(t,s)g(s)f_1(s,u(s))ds,\\  
 L_1u(t):= & \int_a^bk^+(t,s)g(s)u(s)ds.\end{align*}
\end{enumerate}

\begin{rem} Observe that from conditions $(C_6)$ and $(C_8)$ we know that $\psi_{ij}\in K$ and $M_k$ has positive entries for $k=1,2$. Furthermore, if the $\varphi_{ij}$ are linear functionals defined as integrals with respect to a measure of bounded variation, the properties \eqref{superlin}--\eqref{isublin} are satisfied. 
\end{rem}
\begin{rem} In \cite{gi-pp-ft} Infante and co-authors used the cone
$$\hat{K}=K_0 \cap\{u\in C[0,1]: \alpha[u] \geq 0\}\cap\{u\in C[0,1]: \beta[u] \geq 0\},$$
where $$K_{0}:=\{u\in C[0,1]: \min_{t \in [a,b]}u(t)\geq c \|u\|\}$$ and $\alpha$ and $\beta$ are continuous, linear functionals.
Note that the functions in $K_0$ are positive on the
subset $[a,b]$ but are allowed to change sign in $[0,1]$. The cone $K_{0}$ is similar to
a cone of \emph{non-negative} functions first used by Krasnosel'ski\u\i{}, see~\cite{krzab}, and D.~Guo, see e.g.~\cite{guolak},
has been introduced by Infante and Webb in \cite{gijwjiea} and later been used in a number of papers, see for example~\cite{ac-gi-at-bvp, dfgior1, Fan-Ma, giems,  gi-pp, gijwjmaa, gijwems, nietopim} and references therein.

On the other hand, Webb and Infante \cite{jw-gi-jlmsII} used the cone 
$$
\tilde{K}=\{ u \in P: \min_{t \in [a,b]} u(t) \geq c \|u\|, \; \beta_{i}[u] \geq 0 \;\;\text{for every}\;j\},
$$ where $\beta_{i}$ are continuous, linear functionals. Thus the cone $K$ can be seen as an analogue of the cones $\hat{K}$ and $\tilde{K}$ when \emph{nonlinear} functionals are involved.

\end{rem}
\begin{rem}\label{remti} Condition \eqref{sublin} is some sort of triangle inequality. In particular, it implies a kind of \textit{second triangle inequality}. Indeed, let $u-v,v\in K$, then we have
\begin{displaymath}\varphi_{2j}[u]=\varphi_{2j}[(u-v)+v]\le\varphi_{2j}[u-v]+\varphi_{2j}[v].\end{displaymath}
Hence we obtain
\begin{displaymath}\varphi_{2j}[u]-\varphi_{2j}[v]\le\varphi_{2j}[u-v].\end{displaymath}
Therefore,
\begin{displaymath}\varphi_{2j}[v]=\varphi_{2j}[u-(u-v)]\le\varphi_{2j}[u-v]+\varphi_{2j}[u].\end{displaymath}
Thus,
\begin{displaymath}\varphi_{2j}[v]-\varphi_{2j}[u]\le\varphi_{2j}[u-v],\end{displaymath}
which implies, in particular,
\begin{displaymath}|\varphi_{2j}[u]-\varphi_{2j}[v]|\le\varphi_{2j}[u-v].\end{displaymath}
\end{rem}

\begin{rem}\label{remF2}
It follows from $(C_{10})$ that if $u\in K$, then $u^+,|u|\in K$. Furthermore we have $\varphi_{ij}[F_2u]\geq 0$ for every  $u\in K$, where 
$$
F_2u(t):= \int_{0}^{1} |k(t,s)|\,g(s)\,f_2(s,u(s))ds.
$$
\end{rem}
\begin{rem}
The first part of condition $(C_{10})$ implies that, if $u,v\in K$ satisfy $u|_{[a,b]}=v|_{[a,b]}$, then $\varphi_{1j}[u]=\varphi_{1j}[v]$.
\end{rem}

\begin{lem}The operator  $N_f(u,v)(t)=\int_0^1k(t,s)g(s)f(s,u(s),v(s))ds$ maps $C(I)\times L^\infty(I)$ to $C(I)$ and is compact and continuous.
\end{lem}
\begin{proof}
Fix $(u,v)\in C(I)\times L^\infty(I)$ and let $(t_n)_{n\in\bN}\subset I$ be such that $\lim\limits_{n\to\infty}{(t_n)}=t \in I$. Take $r=\|(u,v)\|:=\|u\|+\|v\|$ and consider 
\begin{displaymath}h_n(s):=k(t_n,s)\,g(s)\,f(s,u(s),v(s)),\ \text{for a.e.}\ s \in I.\end{displaymath}
We have, by $(C_1)$, that
\begin{displaymath}\lim_{n\to\infty}{h_n(s)}=h(s):=k(t,s)\,g(s)\,f(s,u(s),v(s)),\ \text{for a.e.}\ s\in I.
\end{displaymath}
On the other hand, $|h_n|\le\Phi\,g\,\|\phi_r\|_\infty$ for all $n \in \bN$. So, by condition $(C_3)$, the sequence $\{h_n\}$ is uniformly bounded in $L^1(I)$ so, by the Dominated Convergence Theorem, we have $\lim\limits_{n\to\infty} N_f(u,v)(t_n)=N_f(u,v)(t)$ and therefore $N_f(u,v) \in C(I)$.\par

Now we show that $N_f$ is compact. Indeed, let $\tilde{B}\subset C(I)\times L^\infty(I)$ be a bounded set, i.e. there exists $r>0$ such that  $\|(u,v)\| \le r<+\infty$ for all $(u,v) \in \tilde{B}$. 

In order to use the Arzel\`a-Ascoli Theorem, we have to verify that $N_f(\tilde{B})$ is a uniformly bounded and equicontinuous set in $C(I)$.

The uniformly boundedness follows from the fact that, for all $t \in I$, the following inequality holds
$$\begin{array}{lll}
| N_f(u,v)(t)| 
\le \int_0^1|k(t,s)| g(s) f(s,u(s),v(s)) ds 
\le  \int_0^1\Phi(s) g(s)\phi_r(s) ds,
\end{array}
$$
for all $(u,v) \in \tilde{B}$. 

On the other hand, taking into account $(C_1)-(C_3)$ and the Dominated Convergence Theorem, we know that for any $\tau \in I$ given, the following property holds:

$$\begin{array}{lll}
\displaystyle \lim_{t \to \tau}  \int_0^1|k(t,s)-k(\tau,s)| g(s) \phi_r(s)  ds=0.
\end{array}
$$

As a consequence,  for any $\tau \in I$ and $\e >0$, there is $\d(\tau)> 0$ such that, if $|t-\tau| < \d(\tau)$, then, for all $(u,v) \in \tilde{B}$, the following inequalities are fulfilled:

$$\begin{array}{lll}
\left|N_f(u,v)(t)- N_f(u,v)(\tau)\right| 
&\le&  \int_0^1|k(t,s)-k(\tau,s)| g(s) f(s,u(s),v(s)) ds\\
&\le&  \int_0^1|k(t,s)-k(\tau,s)| g(s)\phi_r(s) ds < \e.
\end{array}
$$

Now, $\{(\tau-\d(\tau),\tau+\d(\tau))\}_{\tau\in I}$ is an open covering of $I$. Since $I$ is compact, there exists a finite subcovering of indices $\tau_1,\dots,\tau_k$. 

To deduce the equicontinuity of the set $\tilde{B}$ is enough to take $\d_0=\min\{\d(\tau_1),\dots,\d(\tau_k)\}$.

To show the continuity of operator $N_f$, consider $\{(u_n,v_n)\}$ a convergent sequence in $C(I) \times L^\infty(I)$ to $(u,v) \in C(I) \times L^\infty(I)$. In particular, for a.e. $s \in I$, the sequences $\{u_n(s)\}$ and $\{v_n(s)\}$ converge pointwisely to $u(s)$ and $v(s)$ respectively.

Define $y_n(s)=f(s,u_n(s),v_n(s))$. By Condition $(C_5)$, we know that there is $y(s):=\lim\limits_{n\to\infty}y_{n}(s)$ for a.e. $s  \in I$. Since  $|y_n|\le\|\phi_r\|_\infty$ for all $n \in \bN$, we have that the sequence $\{y_n\}$ is uniformly bounded in $L^\infty(I)$. Now, using that $\Phi g \in L^1(I)$, the Dominated Convergence Theorem ensures that  
\begin{align*}
\lim_{n \to \infty} \int_0^1{\Phi(s) g(s) |y_{n}(s)-y(s)| ds}=0.
\end{align*}

Furthermore, using the inequality
\begin{align*}
| N_f(u_{n},v_{n})(t)-N_f(u,v)(t)| 
&\le \int_0^1{|k(t,s)| g(s) |y_{n}(s)-y(s)| ds}\\ 
&\le \int_0^1{\Phi(s) g(s) |y_{n}(s)-y(s)| ds},
\end{align*}
we deduce that such convergence is uniform in $I$, and the assertion holds.
\end{proof}
\begin{lem}
The operator $T$ defined in \eqref{eqthamm} maps $K$ into $K$ and is continuous and compact.
\end{lem}

\begin{proof}
Take $u\in K$. Then, by $(C_2)$, $(C_4)$ and $(C_5)$, we have
\begin{align*}|Tu(t)|= & \left|Bu(t)+\int_0^1 k(t,s)g(s)f(s,u(s),Du(s))\,ds\right|\\ \le & \sum_{j=1}^{n_2}\left|\d_{2j}(t)\right|\b_{2j}[u]+\int_0^1 \Phi(s)g(s)f(s,u(s),Du(s))\,ds.\end{align*}

Hence, we obtain
\begin{displaymath}\|Tu\|\le\sum_{j=1}^{n_2}\|\d_{2j}\|\b_{2j}[u]+\int_0^1 \Phi(s)g(s)f(s,u(s),Du(s))\,ds.\end{displaymath}

Combining this fact with $(C_2)$, $(C_4)-(C_6)$ and $(C_9)$, for $t\in[a,b]$, we get
\begin{equation*}\begin{aligned}
Tu(t) \geq & \sum_{j=1}^{n_1}\d_{1j}(t)\b_{1j}[u]+c_1\int_0^1 \Phi(s)g(s)f(s,u(s),Du(s))\,ds \\ \ge & c\sum_{j=1}^{n_2}\|\d_{2j}\|\b_{2j}[u]+c\int_0^1 \Phi(s)g(s)f(s,u(s),Du(s))\,ds\ge c\|Tu\|.
\end{aligned}
\end{equation*}%
Furthermore, by  $(C_{10})$, $\varphi_{ij} [Tu]\ge0$.
Hence we have $Tu\in K$. \par
Now, we have that the operator $N_f: C(I)\times L^\infty(I)\to C(I)$ such that $N_f(u,v)(t)=\int_0^1k(t,s)g(s)f(s,u(s),v(s))ds$ is compact. 

Since $D$ is continuous, $\Id\times D$ is also continuous so $N_f\circ(\Id\times D)$ is compact. Since $T$ is the sum of two compact operators, it is compact.
\end
{proof}
\begin{rem}\label{remPK} Similarly, from condition $(C_2)$, we observe here that $F_1$, $F_2$ and $L_1$ map $K$ to $K$. To see this, observe that for all $t \in [a,b]$ and $u\in K$ the following properties hold:
\begin{align*} & F_1u(t):=\int_{a}^{b}k(t,s)g(s)f_1(s,u(s))ds\ge c\int_{a}^{b}\Phi(s)g(s)f_1(s,u(s))ds\ge c\|F_1u\|,\\
 & F_2u(t):=\int_{0}^{1}|k(t,s)|g(s)f_2(s,u(s))ds\ge c\int_{0}^{1}\Phi(s)g(s)f_2(s,u(s))ds\ge c\|F_2u\|,\\
&  L_1u(t):=\int_{a}^{b}k^+(t,s)g(s)u(s)ds\ge c\int_{a}^{b}\Phi(s)g(s)u(s)ds\ge c\|L_1u\|.\end{align*}
Also, $\varphi_{ij}[F_1u],\varphi_{ij}[F_2u],\varphi_{ij}[L_1u]\ge0$ by $(C_{10})$ and Remark \ref{remF2}.\par

On the other hand, $L_1$ maps $P$ to $P$, but also maps $P$ to $K$. The proof goes as above.
\end{rem}

\section{Fixed point index calculations}

The following Lemma summarizes some classical results regarding the fixed point index, for more 
details see~\cite{amann, guolak}. Let $U$ be an open bounded subset of $C(I)$, we denote by $U_K:=U \cap K$, which is an open subset in the topology relative to $K$.
\begin{lem} \label{lemind}
Let $U$ be an open bounded set with $0\in U_{K}$ and
$\overline{U}_{K}\ne K$. Assume that $F:\overline{U}_{K}\to K$ is
a compact map such that $x\neq Fx$ for all $x\in \partial U_{K}$. Then
the fixed point index $i_{K}(F, U_{K})$ has the following properties.
\begin{itemize}
\item[(1)] If there
exists $e\in K\setminus \{0\}$ such that $x\neq Fx+\lambda e$ for
all $x\in \partial U_K$ and all $\lambda
>0$, then $i_{K}(F, U_{K})=0$.
\item[(2)] If 
$\mu x \neq Fx$
for all $x\in
\partial U_K$ and for every $\mu \geq 1$, then $i_{K}(F, U_{K})=1$.
\item[(3)] If $i_K(F,U_K)\ne0$, then $F$ has a fixed point in $U_K$.
\item[(4)] Let $U^{1}$ be open in $X$ with
$\overline{U^{1}}\subset U_K$. If $i_{K}(F, U_{K})=1$ and
$i_{K}(F, U_{K}^{1})=0$, then $F$ has a fixed point in
$U_{K}\setminus \overline{U_{K}^{1}}$. The same result holds if
$i_{K}(F, U_{K})=0$ and $i_{K}(F, U_{K}^{1})=1$.

\end{itemize}
\end{lem}
For $\rho>0$ we define the following open subsets of $K$:
\begin{equation*}
  K_{\rho}:=\{u \in K: \|u\|<{\rho}\}, \;\;
  V_{\rho}:=\{u \in K: \min_{t\in[a,b]}u(t)<{\rho}\}.
\end{equation*}
The set $V_\rho$ was introduced in \cite{gijwems}
and is equal to the set called $\Omega_{\rho /c}$ in \cite{gijwjmaa}. The inclusions $$K_{\rho}\subset V_{\rho}\subset K_{\rho/c}$$ play a key role in our existence and multiplicity results.

If $u,\, v$ are vectors, we denote by $[u]_j$ the $j$-th component of $u$ and if we write $u\le v$ the inequality is to be interpreted component-wise. Also, we denote by $\cK_{\varphi_i}:=\(\cK_{\varphi_{ij}}\)_{j=1}^{m_i+n_i},\ i=1,2$ ($\cK_{\varphi_{ij}}$ as defined in $(C_7)$).

The following Lemma gives a sufficient condition that implies that the index is 1.

\begin{lem}
\label{thmind1}
Assume that 
\begin{enumerate}
  \item[$(\mathrm{I}_{\rho }^{1})$]
  there exists $\rho>0$ such that
\begin{align}\label{eqind1b}
  f_2^{-\rho,\rho}\cdot \sup_{t\in I}\(\sum_{j=1}^{m_2+n_2}|\psi_{2j}(t)|\left[(\Id-M_2)^{-1}\int_0^1\cK_{\varphi_2}(s)g(s)ds\right]_j+\sigma(t)\)<1,
\end{align}{}
where
\begin{equation*}
f_2^{-\rho,{\rho}}:=\esssup \Bigl\{ \frac{f_2(t,u)}{ \rho}:\; (t,u)\in
 I \times[-\rho, \rho]\Bigr\}
 \end{equation*}{}
 and
 \begin{displaymath} \sigma(t):=\int_0^1|k(t,s)|g(s)ds.\end{displaymath}

\end{enumerate}
Then we have $i_{K}(T, K_{\rho})=1$.
\end{lem}

\begin{proof}  We show that $Tu \neq \lambda u$ for all $\lambda \geq 1$ when
$ u \in \partial K_{\rho}$, which implies that $i_{K}(T,
K_{\rho})=1$. In fact, if this does not happen, then there exist $u \in K$ with
$\|u\|=\rho$ and $\lambda \geq 1$ such that $ \lambda u(t)=Tu(t)
$.
Therefore, by $(C_4)$ and $(C_5)$,
\begin{equation}
\label{equaub}
 \lambda u(t)\le\sum_{j=1}^{m_2+n_2}\psi_{2j}(t)\varphi_{2j}[u]+F_2u(t),\ t\in I,
 \end{equation}
so, from $(C_6)$ and Remark \ref{remPK}, we have that both sides of the inequality are in $K$. As a consequence, from~\eqref{sublin}, \eqref{e-phi-lambda} and $(C_{10})$, we deduce
\begin{displaymath}\varphi_{2i}[u] \le \varphi_{2i}[\lambda\,u]\le\sum_{j=1}^{m_2+n
_2}\varphi_{2i}[\psi_{2j}]\varphi_{2j}[u]+\varphi_{2i}[F_2u],
\end{displaymath}
which, expressed in matrix notation, is
\begin{displaymath}\varphi_2[u]\le M_2\varphi_2[u]+\varphi_2[F_2u].\end{displaymath}
Hence, we have
\begin{displaymath}(\Id-M_2)\varphi_2[u]\le\varphi_2[F_2u].\end{displaymath}
Since $r(M_2)<1$, $\Id-M_2$ is invertible and $(\Id-M_2)^{-1}=\sum_{k=0}^\infty M_2^k$. Hence, $(\Id-M_2)^{-1}$ is positive and thus, due to the nonnegativeness of  $\varphi_2[F_2u]$, we deduce that
\begin{equation}\label{sold1b}\varphi_2[u]\le(\Id-M_2)^{-1}\varphi_2[F_2u].\end{equation}
Now, for all $t\in I$, using \eqref{isublin}, we have that
\begin{align*}\lambda|u(t)|= &|Tu(t)|=\left|Bu(t)+\int_0^1k(t,s)g(s)f(s,u(s),Du(s))ds\right| \\
 \le &  \left|Bu(t)\right|+\int_0^1|k(t,s)|g(s)f(s,u(s),Du(s))ds \\
 \le & \sum_{j=1}^{n_2}|\delta_{2j}(t)|\beta_{2j}[u]+\int_0^1|k(t,s)|g(s)\left[\sum_{j=1}^{m_2}\gamma_{2j}(s)\alpha_{2j}[u]+f_2(s,u(s))\right]ds\\ 
 = & \sum_{j=1}^{n_2}|\delta_{2j}(t)|\beta_{2j}[u]+\sum_{j=1}^{m_2}\int_0^1|k(t,s)|g(s)\gamma_{2j}(s)ds\,\alpha_{2j}[u]+\int_0^1|k(t,s)|g(s)f_2(s,u(s))ds\\ 
  = & \sum_{j=1}^{n_2}|\delta_{2j}(t)|\beta_{2j}[u]+\sum_{j=1}^{m_2}\til\gamma_{2j}(t)\alpha_{2j}[u]+F_2u(t) \le
  \sum_{j=1}^{m_2+n_2}|\psi_{2j}(t)|\varphi_{2j}[u]+F_2u(t) \\ \le & \sum_{j=1}^{m_2+n_2}|\psi_{2j}(t)|\left[(\Id-M_2)^{-1}\varphi_2[F_2u]\right]_j+F_2u(t)\\
\le & \sum_{j=1}^{m_2+n_2}|\psi_{2j}(t)|\left[(\Id-M_2)^{-1}\int_0^1\varphi_2[|k(t,s)|]g(s)f_2(s,u(s))ds\right]_j+F_2u(t)\\
\le & \sum_{j=1}^{m_2+n_2}|\psi_{2j}(t)|\left[(\Id-M_2)^{-1}\int_0^1\cK_{\varphi_2}(s)g(s)\rho f_2^{-\rho,\rho}ds\right]_j+\int_0^1|k(t,s)|g(s)\rho f_2^{-\rho,\rho}ds\\
\le & \rho f_2^{-\rho,\rho}\cdot \sup_{t\in I}\(\sum_{j=1}^{m_2+n_2}|\psi_{2j}(t)|\left[(\Id-M_2)^{-1}\int_0^1\cK_{\varphi_2}(s)g(s)ds\right]_j+\sigma(t)\).
\end{align*}

Taking the supremum on $t\in I$,
\begin{displaymath}\l\rho\le\rho f_2^{-\rho,\rho}\cdot \sup_{t\in I}\(\sum_{j=
1}^{m_2+n_2}|\psi_{2j}(t)|\left[(\Id-M_2)^{-1}\int_0^1|\cK_{\varphi_2}(s)|g(s)ds\right]_j+\sigma(t)\).\end{displaymath}
From \eqref{eqind1b} we obtain $\lambda \rho <\rho$, contradicting the fact that $\lambda\geq 1$.
\end{proof}

\begin{rem}
We point out, in similar way as in \cite{jwgi-nodea-08}, that a stronger (but easier to check) condition than $(\mathrm{I}_{\protect\rho }^{1})$ is given by the following.
 \begin{displaymath}f_2^{-\rho,\rho}\(\sum_{j=1}^{m_2+n_2}\|\psi_{2j}\|\left[(\Id-M_2)^{-1}\int_0^1\left|\cK_{\varphi_2}(s)\right| g(s)ds\right]_j+\frac{1}{m}\)<1.\end{displaymath}
where
\begin{equation}\label{mdef}\frac{1}{m}:=\sup_{t\in I}\sigma(t).\end{equation}
A similar constant has been considered in~\cite{jwgi-nodea-08 , ac-gi-at-bvp, ac-gi-at-tmna, gijwjiea}.
\end{rem}
The next Lemma yields a condition sufficient for the index to be 0.
\begin{lem}
Assume that \par
\begin{enumerate}
\item[$(\mathrm{I}_{\protect\rho }^{0})$]
There exists $\rho>0$ such that
\[
 f_{1,\rho,\rho/c}\cdot \inf_{t\in [a,b]}\(\sum_{j=1}^{m_1+n_1}\psi_{1j}(t)\left[(\Id-c_1M_1)^{-1}\int_a^b\cK_{\varphi_1}(s)g(s)ds\right]_j+\int_a^bk(t,s)g(s)ds\)>1,
\]
where
\begin{equation*}
f_{1,\rho,\rho/c}: =\essinf \left\{\frac{f_1(t,u)}{\rho }%
:\;(t,u)\in [a,b]\times [\rho ,\rho /c]\right\}.
\end{equation*}{}
\end{enumerate}
Then we have $i_{K}(T, V_{\rho})=0.$
\end{lem}
\begin{proof}
 Take $e\in K\backslash\{0\}$ (for instance $e=\til\gamma_{21}$).
We will show that $u \neq Tu +{\lambda}e$ for all
${\lambda} \geq 0$
 and $u \in \partial V_{\rho}$ which implies that $i_{K}(T, V_{\rho})=0$.
In fact, if this does not happen, there are $u \in \partial V_{\rho}$ (and so  we have $\min_{t \in [a,b]} {u(t)}=\rho$ and $\rho
\leq u(t) \leq \rho/c$  for all $t \in [a,b]$), and ${\lambda} \geq 0$ with
\begin{equation*}
u(t)=Tu(t)+{\lambda}e.
\end{equation*}
Therefore, for $t\in[a,b]$, by $(C_2)$, $(C_4) - (C_6)$ and Remark \ref{remPK}, we have
\begin{equation}\label{equaub2} u(t)\ge\sum_{j=1}^{m_1+n_1}\psi_{1j}(t)\varphi_{1j}[u]+F_1u(t)+\l \,e(t).
\end{equation}
Thus, using again $(C_6)$, $(C_7)$ and $(C_{10})$ together with \eqref{superlin}, we obtain
\begin{displaymath}\varphi_{1i}[u]\ge\sum_{j=1}^{m_1+n_1}\varphi_{1i}[\psi_{1j}]\varphi_{1j}[u]+\varphi_{1i}[F_1u]+\l\varphi_{1i}[e]\ge
c_1 \, \left(\sum_{j=1}^{m_1+n_1}\varphi_{1i}[\psi_{1j}]\varphi_{1j}[u]+\varphi_{1i}[F_1u]\right),\end{displaymath}
which, expressed in matrix notation, is
\begin{displaymath}\varphi_1[u]\ge c_1 \, \left(M_1\varphi_1[u]+\varphi_1[F_1u]\right).\end{displaymath}
Hence we get
\begin{displaymath}(\Id-c_1 M_1)\varphi_1[u]\ge\varphi_1[F_1u].\end{displaymath}
Since $r(M_1)<1/c_1$, $\Id-c_1\,M_1$ is invertible and \[(\Id-c_1\,M_1)^{-1}=\sum_{k=0}^\infty \left({c_1\,M_1}\right)^k,\] so $(\Id-c_1 M_1)^{-1}$ is positive and hence
\begin{equation}\label{sold2b}\varphi_1[u]\ge(\Id-c_1\, M_1)^{-1}\varphi_1[F_1u].\end{equation}
Therefore, from \eqref{equaub2} and \eqref{sold2b} we obtain, using \eqref{isublin}, for $t\in[a,b]$,
\begin{align*}u(t)\ge & \sum_{j=1}^{m_1+n_1}\psi_{1j}(t)\varphi_{1j}[u]+F_1u(t)\ge\sum_{j=1}^{m_1+n_1}\psi_{1j}(t)\left[(\Id-c_1\,M_1)^{-1}\varphi_1[F_1u]\right]_j+F_1u(t)\\
\ge &\rho f_{1,\rho,\rho/c}\inf_{t\in [a,b]}\(\sum_{j=1}^{m_1+n_1}\psi_{1j}(t)\left[(\Id-c_1\,M_1)^{-1}\int_
a^b\cK_{\varphi_1}(s)g(s)ds\right]_j+\int_a^bk(t,s)g(s)ds\).
\end{align*}
Taking the infimum on $t\in [a,b]$, gives
\begin{displaymath}\rho\ge\rho f_{1,\rho,\rho/c}\inf_{t\in [a,b]}\(\sum_{j=1}^{m_1+n_1}\psi_{1j}(t)\left[(\Id-c_1\,M_1)^{-1}\int_a^b\cK_{\varphi_1}(s)g(s)ds\right]_j+\int_a^bk(t,s)g(s)ds\).\end{displaymath}
which contradicts the hypothesis.
 \end{proof}
\begin{rem}
We point out, in similar way as in \cite{jwgi-nodea-08}, that a stronger (but easier to check) condition than $(\mathrm{I}_{\protect\rho }^{0})$ is given by the following.
\begin{equation*}
f_{1,\rho,\rho/c}\(\inf_{t\in [a,b]}\sum_{j=1}^{m_1+n_1}\psi_{1j}(t)\left[(\Id-c_1M_1)^{-1}\int_a^b\cK_{\varphi_1}(s)g(s)ds\right]_j+\frac{1}{M(a,b)}\)>1,
\end{equation*}{}
where
\begin{equation}\label{Mdef} 
\frac{1}{M(a,b)} :=\inf_{t\in [a,b]}\int_{a}^{b}k(t,s)g(s)\,ds.
\end{equation}
\end{rem}
The results above can be used in order to prove the existence of at least one, two or three nontrivial solutions. We omit the proof  which follows from the properties of the fixed point index. We note that, by expanding the lists in conditions $(S_{5}),(S_{6})$ below, it is
possible to state results for four or more nontrivial solutions, see for
example the paper by Lan~\cite{kljdeds} for the type of results that might be stated.

\begin{thm}\label{thmcasesS}Assume conditions $(C_1)-(C_{10})$ are satisfied. The integral equation \eqref{eqthamm} has at least one non-zero solution
in $K$ if one of the following conditions hold.
\begin{enumerate}
\item[$(S_{1})$] There exist $\rho _{1},\rho _{2}\in (0,\infty )$ with $\rho
_{1}/c<\rho _{2}$ such that $(\mathrm{I}_{\rho _{1}}^{0})$ and $(\mathrm{I}_{\rho _{2}}^{1})$ hold.
\item[$(S_{2})$] There exist $\rho _{1},\rho _{2}\in (0,\infty )$ with $\rho
_{1}<\rho _{2}$ such that $(\mathrm{I}_{\rho _{1}}^{1})$ and $(\mathrm{I}%
_{\rho _{2}}^{0})$ hold.
\end{enumerate}
The integral equation \eqref{eqthamm} has at least two non-zero solutions in $K$ if one of
the following conditions hold.
\begin{enumerate}
\item[$(S_{3})$] There exist $\rho _{1},\rho _{2},\rho _{3}\in (0,\infty )$
with $\rho _{1}/c<\rho _{2}<\rho _{3}$ such that $(\mathrm{I}_{\rho
_{1}}^{0}),$ $(
\mathrm{I}_{\rho _{2}}^{1})$ $\text{and}\;\;(\mathrm{I}_{\rho _{3}}^{0})$
hold.
\item[$(S_{4})$] There exist $\rho _{1},\rho _{2},\rho _{3}\in (0,\infty )$
with $\rho _{1}<\rho _{2}$ and $\rho _{2}/c<\rho _{3}$ such that $(\mathrm{I}%
_{\rho _{1}}^{1}),\;\;(\mathrm{I}_{\rho _{2}}^{0})$ $\text{and}\;\;(\mathrm{I%
}_{\rho _{3}}^{1})$ hold.
\end{enumerate}
The integral equation \eqref{eqthamm} has at least three non-zero solutions in $K$ if one
of the following conditions hold.
\begin{enumerate}
\item[$(S_{5})$] There exist $\rho _{1},\rho _{2},\rho _{3},\rho _{4}\in
(0,\infty )$ with $\rho _{1}/c<\rho _{2}<\rho _{3}$ and $\rho _{3}/c<\rho
_{4}$ such that $(\mathrm{I}_{\rho _{1}}^{0}),$ $(\mathrm{I}_{\rho _{2}}^{1}),\;\;(\mathrm{I}%
_{\rho _{3}}^{0})\;\;\text{and}\;\;(\mathrm{I}_{\rho _{4}}^{1})$ hold.
\item[$(S_{6})$] There exist $\rho _{1},\rho _{2},\rho _{3},\rho _{4}\in
(0,\infty )$ with $\rho _{1}<\rho _{2}$ and $\rho _{2}/c<\rho _{3}<\rho _{4}$
such that $(\mathrm{I}_{\rho 
_{1}}^{1}),\;\;(\mathrm{I}_{\rho
_{2}}^{0}),\;\;(\mathrm{I}_{\rho _{3}}^{1})$ $\text{and}\;\;(\mathrm{I}%
_{\rho _{4}}^{0})$ hold.
\end{enumerate}
\end{thm}

\subsection{Non-existence results}
For this epigraph we will assume that the operators $\varphi_{ij}$ are \textit{linearly bounded} i.\,e., an operator $A:X\to Y$ between two normed spaces $X$ and $Y$ is linearly bounded if there exists $M\in\bR^+$ such that $\|Ax\|\le M\|x\|$ for every $x\in\ X$. We define the \textit{norm} of $A$ as $\|A\|:=\inf\{M\in\bR^+\ :\  \|Ax\|\le M\|x\|,\ x\in X\}$. Observe that for linear operators this is the usual norm. We denote by $\LB(X,Y)$ the space of linearly bounded operators from $X$ to $Y$ (and by $\LB(X)$ if $X=Y$). \par
We now offer some non-existence results for the integral equation~\eqref{eqthamm}.

\begin{thm} Assume conditions $(C_1)-(C_5)$ are satisfied. Let $m$  be as in \eqref{mdef} and $M(a,b)$ as in \eqref{Mdef}. If one of the following conditions holds, 
\begin{enumerate}
\item $f_2(t,u)<m\(1-\sum_{j=1}^{m_2+n_2}\|\psi_{2j}\|\|\varphi_{2j}\|\)|u|$, for every $t\in I$ and $u\in\bR\backslash\{0\}$,
\item $f_1(t,u)>M(a,b)\,u$ for every $t\in [a,b]$ and $u\in\bR^+$,
\end{enumerate}
then there is no non-trivial solution of the integral equation~\eqref{eqthamm} in $K$.
\end{thm}
\begin{proof}
$(1)$ Assume, on the contrary, that there exists $u\in K$, $u\not\equiv 0$ such that $u=Tu$ and let $t_0\in I$ such that $\|u\|=|u(t_0)|$. Then we have
\begin{align*}  \|u\|  = &  |u(t_0)| \\ \le & \sum_{j=1}^{m_2+n_2}\|\psi_{2j}\|\varphi_{2j}[u]+\int_{0}^{1}|k(t_0,s)|g(s)f_2(s,u(s))\,ds \\ < & \sum_{j=1}^{m_2+n_2}\|\psi_{2j}\|\|\varphi_{2j}\|\|u\|+\int_{0}^{1}|k(t_0,s)|g(s)\,ds\,m\(1-\sum_{j=1}^{m_2+n_2}\|\psi_{2j}\|\|\varphi_{2j}\|\)\|u\|\\
\le & \sum_{j=1}^{m_2+n_2}\|\psi_{2j}\|\|\varphi_{2j}\|\|u\|+\(1-\sum_{j=1}^{m_2+n_2}\|\psi_{2j}\|\|\varphi_{2j}\|\)\|u\|=\|u\|,
\end{align*}
a contradiction, thus there is no non-trivial solution  of the integral equation~\eqref{eqthamm} in $K$.\par
$(2)$ Assume, on the contrary, that there exists $u\in K$, $u\not\equiv 0$ such that $u=Tu$ and let $t_0\in I$ such that $u(t_0)=\min_{t\in[a,b]}u(t)$. Then,
\begin{align*}u(t_0)=Tu(t_0)\ge & \sum_{j=1}^{m_2+n_2}\psi_{1j}(t_0)\varphi_{1j}[u]+ \int_0^1k(t_0,s)g(s)f_1(s,u(s))ds \\> & \int_a^bk(t_0,s)g(s)M(a,b)u(s)ds\\ \ge &  M(a,b)u(t_0)\int_a^bk(t_0,s)g(s)ds\ge u(t_0),
\end{align*}
a contradiction. Thus there is no non-trivial solution of the integral equation~\eqref{eqthamm} in $K$.
\end{proof}

\section{The spectral radius and the existence of multiple solutions}

In order to prove the results that follow we make use of different requirements on the functionals $\varphi_{ij}$ than being linearly bounded. We introduce now some definitions, see~\cite{Dol1,Dol2}.\par
For operators $A\in\LB(X)$ we can define the \textit{spectral radius} of $A$ as $r(A)=\lim\limits_{n\to\infty}\|A^n\|^\frac{1}{n}$. We define the \textit{principal characteristic value} as $\mu(A):=1/r(A)$. For more properties of this generalized spectral value we refer the reader to \cite{Bug,Zim}.\par
 Let $(X,\|\cdot\|_X),(Y,\|\cdot\|_Y)$ be real normed spaces. Let $\Lip(X,Y)$ 
be the set of operators from $X$ to $Y$ that satisfy the \textit{Lipschitz property}, that is,
\begin{displaymath}\Lip(X,Y):=\{N:X\to Y\ :\ \exists M\in\bR^+, \|Nx-Ny\|_Y\le M\|x-y\|_X,\ x,y\in X\}.\end{displaymath}
Define the function
\begin{displaymath}\|N\|^*:=\inf\{M\in\bR^+\ :\ \|Nx-Ny\|_Y\le M\|x-y\|_X,\ x,y\in X\},\ N\in\Lip(X,Y).\end{displaymath}
We denote by $\Lip(X)\equiv\Lip(X,X)$. $\Lip(X,Y)$ is a real vector space and $\|\cdot\|^*$ is a seminorm on $\Lip(X,Y)$ (in fact, $(\|\cdot\|^*)^{-1}(\{0\})=\bR$). Also, observe that
\begin{displaymath}\|N-N(0)\|=\sup_{\substack{x\in X,\\ x\ne 0}}\frac{\|N(x)-N(0)\|_Y}{\|x\|_X}\le \sup_{\substack{x,y\in X,\\ x\ne y}}\frac{\|N(x)-N(y)\|_Y}{\|x-y\|_X}=\|N\|^*,\end{displaymath}
thus, in particular, $N-N(0)$ is linearly bounded for every $N\in\Lip(X,Y)$. On the other hand if $N(0)\ne 0$, $N$ is not linearly bounded, for the definition of linearly bounded operators implies that they vanish at zero. With these considerations in mind we can define then
\begin{displaymath}\Lip_0(X,Y):=\Lip(X,Y)\cap\LB(X,Y)=\{N\in\Lip(X,Y)\ :\ N(0)=0\}.\end{displaymath} Note that
$\|\cdot\|^*$ is a norm on $\Lip_0(X,Y)$.\par

The following Theorems from \cite{Dol2} characterize invertibility of the operators between $X$ and $Y$.
\begin{thm} \cite[Theorem 1]{Dol2}
Let $X$ a real normed space and $Y$ a real Banach space. Let $N:X\to Y$ be an operator. Then $N$ is invertible if and only if there exists an invertible operator $J:Y\to X$ such that $(N-J)J^{-1}\in\Lip(Y)$ and $\|(N-J)J^{-1}\|^*<1$.
\end{thm}
\begin{thm} \cite[Theorem 2]{Dol2}
Let $X$ a real normed space and $Y$ a real Banach space. Let $N:X\to Y$ be an operator. Then $N$ is invertible and $N\in\Lip(X,Y)$, if and only if there exists an invertible operator $J:Y\to X$ with inverse $J^{-1}\in\Lip(X,Y)$ such that $(N-J)J^{-1}\in\Lip(Y)$ and $\|(N-J)J^{-1}\|^*<1$.\par In such a case, $\|N^{-1}\|^*\le\|J^{-1}\|^*/(1-\|(N-J)J^{-1}\|^*)$.
\end{thm}
The following consequence (in the line of \cite[Corollary 2]{Dol1}) can be obtained by taking $X=Y$, $N=\Id-Q$, $J=\Id$.
\begin{cor}
Let $X$ be a real Banach space and $Q\in\Lip(X)$ such that $\|Q\|^*<1$. Then $\Id-Q$ is an invertible operator and $\|(\Id-Q)^{-1}\|^*\le1/(1-\|Q\|^*)$.
\end{cor}
\begin{rem}\label{rangerem}
Assume $Q\in\Lip(X)$, $Q(X)$ closed for the sum, $\|Q\|^*<1$. Then \begin{displaymath}(\Id-Q)^{-1}|_{Q(X)}:Q(X)\to Q(X).\end{displaymath} 

To see this take $x\in X$ and define $y=(\Id-Q)^{-1}Qx$. Then $y=Qx+Qy\in Q(X)$.
\end{rem}
We now present a result which is a straightforward generalization to the case of linearly bounded operators of a classical result on linear operators.

Let us define the following operators and constants from the functions defined in conditions $(C_{1})$-$(C_{10})$.
\begin{align*}H_1u(t):= & \sum_{j=1}^{m_1+n_1}\psi_{1j}(t)\varphi_{1j}[u], \\
L_2u(t):= & \int_0^1|k(t,s)|g(s)u(s)ds,\ H_2u(t):= \sum_{j=1}^{m_2+n_2}|\psi_{2j}(t)||\varphi_{2j}[u]|,\\
f_2^0:=& \lils{u\to0}\esssup\limits_{t\in I}\frac{f_2(t,u)}{|u|},\ f_{1,0}:= \lili{u\to0^+}\essinf_{t\in [a,b]}\frac{f_1(t,u)}{u},
\\
f_2^\infty:=& \lils{|u|\to\infty}\esssup_{t\in I}\frac{f_2(t,u)}{|u|},\ f_{1,\infty}:= \lili{u\to\infty}\essinf_{t\in [a,b]}\frac{f_1(t,u)}{u}.
\end{align*}

\begin{lem} Assume conditions $(C_1)$-$(C_7)$. Assume also that condition \eqref{sublin} holds for every $u,v\in C(I)$ and $\varphi_{2j}\in\LB(C(I))$, $j=1,\dots,m_2+n_2$, then
$H_2\in\Lip_0(C(I))$.
\end{lem}
\begin{proof}Let $u,v\in C(I)$. Using inequality \eqref{sublin} and Remark \ref{remti} we obtain
\begin{align*}|H_2u-H_2v|=& \left|\sum_{j=1}^{m_2+n_2}|\psi_{2j}(t)|\varphi_{2j}[u]-\sum_{j=1}^{m_2+n_2}|\psi_{2j}(t)|\varphi_{2j}[v]\right|=\left|\sum_{j=1}^{m_2+n_2}|\psi_{2j}(t)|\(\varphi_{2j}[u]-\varphi_{2j}[v]\)\right|\\ \le & \sum_{j=1}^{m_2+n_2}\|\psi_{2j}\|\left|\varphi_{2j}[u]-\varphi_{2j}[v]\right|\le\sum_{j=1}^{m_2+n_2}\|\psi_{2j}\|\left|\varphi_{2j}[u-v]\right| \\ \le & \sum_{j=1}^{m_2+n_2}\|\psi_{2j}\|\|\varphi_{2j}\|\|u-v\|.\end{align*}
Hence, $H_2\in\Lip(C(I))$ and $\|H_2\|^*\le\sum_{j=1}^{m_2+n_2}\|\psi_{2j}\|\|\varphi_{2j}\|$. Also, since $H_2\in\LB(C(I))$, $H_2(0)=0$, so $H_2\in\Lip_0(C(I))$.
\end{proof}
We now recall the celebrated Krein-Rutman theorem.
\begin{thm}[Theorem 19.2 and Ex.~12 of~\cite{Deimling}] Let $X$ be a Banach space, $K\subset X$ a total cone, that is, $\ol{K-K}=X$, and $L:X\to X$ a continuous compact linear operator that maps $K$ to $K$ with positive spectral radius $r(L)$. Then $r(L)$ is an eigenvalue of $L$ with an eigenfunction in $K\backslash\{0\}$.
\end{thm}
\begin{cor}\label{coreigen}
The spectral radius of $L_1$ is an eigenvalue of $L_1$ with an eigenfunction in $P\cap K$.
\end{cor}
\begin{proof}
Recall that $L_1$ is continuous, compact and maps $P$ to $P\cap K$ (see Remark \ref{remPK}).  Also, $P$ is a total cone. Let $u\in P$, $u\equiv 1$ in $[a,b]$. $L_1u(t)$ in $[0,1]$ does not depend on the values of $u$ in $[0,1]\backslash[a,b]$, and in particular  we have
\begin{displaymath}
L_1u(t)\equiv h(t)=\int_a^bk^+(t,s)g(s)ds\ge c\int_a^b\Phi(s)g(s)ds=:q,\ t\in [a,b].
\end{displaymath}
Assume $a>0$ and $b<1$ (in other cases it is straightforward). Since $h$ is a continuous function in $[0,1]$, there are some $\hat a,\hat b\in(0,1)$ such that $\hat a< a$ and $\hat b>b$ satisfying
\begin{displaymath}
h(t)>\frac{q}{2},\  t\in [\hat a,\hat b].
\end{displaymath}
Hence, defining
\[
u(t)=
\left\{
\begin{array}{lll}
0 & \mbox{if} & 0, \le t \le \hat a,\\
\displaystyle \frac{t-\hat a}{a-\hat a}, & \mbox{if} & \hat a \le t \le  a,\\
1, & \mbox{if} & a \le t \le b,\\
\displaystyle \frac{t
-\hat b}{b-\hat b}, & \mbox{if} & b \le t \le  \hat b,\\
0, & \mbox{if} & \hat b \le t \le 1,
\end{array}
\right.
\]
it can be verified that $u\in P$ and 
$L_1u\ge\l_0u$ on $[0,1]$ for  $\l_0=q/2$. Hence, by iteration, we have $L_1^nu\ge\l_0^nu$ for all $n\in\bN$ and therefeore $\|L_1^n\|\ge\|L_1^nu\|\ge\l_0^n\|u\|=\l_0^n$. Thus we have
\[r(L_1)=\lim_{n\to\infty}\|L_1^n\|^\frac{1}{n}\ge\l_0.\]

Therefore, the hypotheses of the Krein-Rutman Theorem are satisfied and, as consequence, there exists $\upsilon\in P$ such that $L_1\upsilon=r(L_1)\upsilon$. Since $L_1:P\to P\cap K$, we know that $\upsilon\in P\cap K$.
\end{proof}
In order to prove the next result, we use the following operator on $C[a,b]$ defined by
\begin{displaymath}
\bar{L} u(t):=  \int_a^b k^+(t,s)g(s)u(s)\,ds,\ t\in[a,b]
\end{displaymath}
and the cone $P_{[a,b]}$ of positive functions in $C[a,b]$. 
 
In the recent papers \cite{jw-lms, jw-tmna}, Webb developed an elegant theory valid for $u_0$-positive linear operators relative to two cones. It turns out that our operator  $\bar{L}$ fits within this setting and, in particular, satisfies the assumptions of Theorem $3.4$ of \cite{jw-tmna}.  We state here a special case of Theorem $3.4$ of \cite{jw-tmna} that can be used for $\bar{L}$.
 
\begin{thm}\label{thmjeff}
Suppose that there exist $u\in P_{[a,b]}\setminus \{0\}$ and $\lambda>0$ such that \begin{displaymath}\lambda u(t)\geq \bar{L}u(t),\ \text{for}\ t\in [a,b].\end{displaymath} Then we have 
$r(\bar{L})\leq \lambda$.

\end{thm}
\begin{thm}\label{thmindeig}
Assume conditions $(C_1)$ - $(C_6)$,  $\varphi_{2j}[u]\ge\varphi_{2j}[v]$ for every $u,v\in K$ such that $u(t)\ge v(t)$ for all $t\in I$ and $\varphi_{ij}[u]\ge0$ for every $u\in P$ (part of $(C_{10})$. We have the following.
\begin{enumerate}
\item  If $H_2\in\Lip_0(C(I))$, $\|H_2\|^*<1$, $(\Id-H_2)^{-1}L_2\in\LB(C(I))$, $(\Id-H_2)^{-1}:K\cap P\to K\cap P$ is order preserving, $(\Id-H_2)^{-1}(\l u)\le\l (\Id-H_2)^{-1}u$ for every $\l\in\bR^+$, $u\in K\cap P$ and  
 $0\le f_2^{0}<\mu((\Id-H_2)^{-1}L_2)$, 
 then there exists $\rho_{0}>0$ such
that
\begin{equation*}
i_{K}(T,K_{\rho})=1 \; \text{ for each } \; \rho\in (0,\rho_{0}].
\end{equation*}{}
\item If $\mu( L_1)<f_{1,0}\leq \infty$, then there exists
$\rho_{0}>0$ such that for each $\rho\in (0,\rho_{0}]$
\begin{equation*}
i_{K}(T,K_{\rho})=0.
\end{equation*}{}
\item If $\mu( L_1)<f_{1,\infty} \leq \infty$, then there
exists $R_{1}$ such that for each $R \geq R_{1}$
\begin{equation*}
i_{K}(T,K_{R})=0.
\end{equation*}{}
\end{enumerate}
\end{thm}

\begin{proof}

$(1)$
 Let $\xi=\mu((\Id-H_2)^{-1}L_2)$. By the hypothesis, there exist $\rho_0,\tau\in(0,1)$ such that
\begin{displaymath}f_2(t,u)\le(\xi-\tau)|u|\end{displaymath}
for all $u\in[-\rho_0,\rho_0]$ and almost every $t\in I$.\par

Let $\rho\in(0,\rho_0]$. We prove that $Tu\ne\lambda u$ for $u\in\partial K_\rho$ and $\lambda\ge 1$, which implies the result by Lemma \ref{lemind}. In fact, if we assume otherwise, then there exists $u\in\partial K_\rho$ and $\lambda\ge1$ such that $\l u=Tu$. Observe that if $u\in K$, using what is assumed of $(C_{10})$, we conclude that $|u|\in K\cap P$ and for $t\in I$,
\begin{align*}|u(t)|\le & \l |u(t)|=  |Tu(t)|  \le  H_2u(t)+\int_0^1|k(t,s)|g(s)f_2(s,u(s))ds\\ \le & H_2|u|(t)+ (\xi-\tau) L_2|u|(t).
 \end{align*}
 Now we have
 \begin{displaymath}|u|(t)\le(\Id-H_2)^{-1}(\xi-\tau)L_2|u|(t)\le(\xi-\tau)(\Id-H_2)^{-1}L_2|u|(t).\end{displaymath}
 Iterating, that is, substituting the LHS into the RHS, for $n\in\bN$, we obtain
 \begin{displaymath}|u|(t)\le\dots\le\left[(\xi-\tau)(\Id-H_2)^{-1}L_2\right]^n|u|(t).\end{displaymath}
 Therefore, taking norms, we have
  \begin{displaymath}\|u\|\le\|\left[(\xi-\tau)(\Id-H_2)^{-1}L_2\right]^n|u|\|,\end{displaymath}
  which implies
   \begin{displaymath}1\le\|\left[(\xi-\tau)(\Id-H_2)^{-1}L_2\right]^n\|,\end{displaymath}
   or
    \begin{displaymath}1\le(\xi-\tau)\|\left[(\Id-H_2)^{-1}L_2\right]^n\|^\frac{1}{n}.\end{displaymath}
Taking the limit both sides we arrive to a contradiction, \begin{displaymath}1\le\frac{\xi-\tau}{\xi}<1.\end{displaymath}
$(2)$ There exists $\rho_0>0$ such that $f_1(t,u)\ge\mu(L_1)u$ for all $u\in[0,\rho_0]$ and almost all $t\in [a,b]$. Let $\rho\in[0,\rho_0]$. Let us prove that
$u\ne Tu+\lambda\upsilon_1$ for all $u$ in $\partial K_\rho$ and $\lambda\ge0$,
where $\upsilon_1\in K$ is the eigenfunction of $L_1$ with $\|\upsilon_1\|=1$ corresponding to the eigenvalue $1/\mu(L_1)$, which would imply the result (cf. Corollary \ref{coreigen}).\par
We distinguish now two cases, $\l\in\bR^+$ and $\l=0$. Assume, on the contrary, that there exist $u\in\partial K_\rho$ and $\lambda\in\bR^+$ such that $u=Tu+\lambda\upsilon_1$. Since $Tu\ge0$ in $[a,b]$, we have that $u\ge\lambda\upsilon_1$ in $[a,b]$ and $L_1u\ge\lambda L_1\upsilon_1=[\lambda/\mu(L_1)]\upsilon
_1$ in $[a,b]$. Using this and the previous estimate for $f$ we have, by $(C_4)$ and $(C_6)$, in $[a,b]$,
\begin{align*}
Tu(t)= & Bu(t)+\int_0^1 k(t,s)g(s)f(s,u(s),Du(s))\,ds\ge \int_0^1 k(t,s)g(s)f_1(s,u(s))\,ds \\ \ge & \mu(L_1)\int_a^b k^+(t,s)g(s)u(s)\,ds=\mu(L_1)\, L_1u(t),
\end{align*}
so
\begin{displaymath}u\ge\mu(L_1) L_1u+\lambda\upsilon_1\ge\lambda \mu(L_1) L_1 \upsilon_1+\lambda\upsilon_1=2\lambda\upsilon_1,\ \text{in }[a,b].\end{displaymath}

Through induction we deduce that $\rho\ge u\ge n\lambda\upsilon_1$ in $[a,b]$ for every $n\in\bN$, a contradiction because $\upsilon_1\in K\backslash\{0\}$.\par
Now we consider the case $\lambda=0$.  Let $\varepsilon>0$ be such that for all $u\in[0,\rho_0]$ and almost every $t\in [a,b]$ 
we have
\begin{equation*}
f_1(t,u)\geq (\mu( L_1)+\varepsilon)u.
\end{equation*}
We have, for $t\in [a,b]$, 
\begin{equation*}
 u(t) \geq(\mu(L_1)+\varepsilon) L_1u(t).
\end{equation*}

Since $ L_1\upsilon_1(t)=r( L_1)\upsilon_1(t)$ for $t\in[0,1]$, we have, for $t\in[a,b]$,
\begin{displaymath}
\bar{L} \upsilon_1(t)=L_1\upsilon_1(t)=r(L_1)\upsilon_1(t),
\end{displaymath}
and we obtain $r(\bar{L})\geq r(L_1)$. On the other hand, we have, for $t\in [a,b]$, 
\begin{align*}
 u(t)= & Tu(t)= Bu(t)+\int_0^1 k(t,s)g(s)f(s,u(s),Du(s))\,ds\\ \ge & (\mu(L_1)+\varepsilon)\int_a^b k(t,s)g(s)u(s)\,ds = (\mu(L_1)+\varepsilon) L_1u(t)=(\mu(L_1)+\varepsilon) \bar{L} u(t).
\end{align*}
where $u(t)>0$ in $[a,b]$. Thus, using Theorem~\ref{thmjeff}, we have  $r(\bar{L})\leq 1/(\mu(L_1)+\varepsilon)$ and therefore $r(L_1)\leq 1/(\mu(L_1)+\varepsilon)$. This gives $\mu(L_1)+\varepsilon\leq \mu(L_1)$, a contradiction. \par
\par
$(3)$ Take $v_
1$ as in part (2). Let $R_1\in\bR^+$ such that $f_1(t,u)>\mu(L_1)u$ for all $u\ge c R_1$, $c$ as in $(C_4)$, and almost all $t\in [a,b]$. We will prove that
$u\ne Tu+\lambda\upsilon_1$ for all $u$ in $\partial K_R$ and $\lambda\in\bR^+$ when $R>R_1$. Observe that for $u\in\partial K_R$, we have $u(t)\ge c\|u\|\ge c R_1$ for all $t \in [a,b]$, so $f_1(t,u(t))>\mu(L_1)u(t)$ for a.e.  $t \in [a,b]$.
\par
Assume now, on the contrary, that there exist $u\in\partial K_R$ and $\l\in\bR^+$ (the proof in the case $\lambda=0$ is treated as in the proof of the statement~$(2)$) such that $u=Tu+\lambda\upsilon_1$. This implies $u\ge\lambda\upsilon_1$ in $[a,b]$ and $L_1u\ge\lambda L_1\upsilon_1=[\lambda/\mu(L_1)]\upsilon_1$ in $[a,b]$. Using this and the previous estimate for $f$ we have
\begin{displaymath}u\ge\mu(L_1) L_1u+\lambda\upsilon_1\ge\lambda \mu(L_1) L_1 \upsilon_1+\lambda\upsilon_1=2\lambda\upsilon_1,\ \text{in }[a,b].\end{displaymath}
Through induction  we deduce that $R\ge u\ge n\lambda\upsilon_1$ for every $n\in\bN$, a contradiction because $\upsilon_1\in K\backslash\{0\}$.
\end{proof}
\begin{rem}\label{remp1}
In the previous Theorem, in point $(1)$, it is enough to ask for $L_2\in\LB(C(I))$ in order to have $(\Id-H_2)^{-1}L_2\in\LB(C(I))$ since $(\Id-H_2)^{-1}\in\Lip(C(I))$.
\end{rem}
\begin{rem} It is clear that the spectral radius of a linearly bounded operator is bounded from above by the norm $\|\cdot\|$. Hence, in the previous Theorem, in point $(
1)$ the condition  $0\le f_2^{0}<\mu((\Id-H_2)^{-1}L_2)$ can be strengthened to $0\le f_2^{0}<1/\|(\Id-H_2)^{-1}L_2\|$. Furthermore, if $L_2\in\LB(C(I))$, we can strengthen it even further  to $0\le f_2^{0}<(1-\|H_2\|^*)/\|L_2\|$.
\end{rem}

\begin{rem}\label{remp2} In the previous Theorem, the conditions $\mu(L_1)<f_{1,0} \leq \infty$ and $\mu( L_1)<f_{1,\infty} \leq \infty$ in $(2)$ and $(3)$ respectively can be strengthen in order to avoid the computation of the spectral value of $L_1$. As it is shown in \cite{jwkleig}, the new conditions would be \[1/\inf_{t\in[a,b]}\int_a^bk(t,s)g(s)ds<f_{1,0} \leq \infty,\] and \[1/\inf_{t\in[a,b]}\int_a^bk(t,s)g(s)ds<f_{1,\infty} \leq \infty.\]

\end{rem}
\section{An application}
In order to prove the usefulness of our theory, we present a simple but yet fairly general application in this Section.
Consider the BVP
\begin{equation}\label{eqpull}-u''(t)=f(t,u(t))+\c(t)u(\eta(t)),\ t\in[0,1],\quad u(0)=u(1)=\theta\max_{t\in[a,b]} u(t).\end{equation}
where $f$ satisfies the $L^\infty$-Carath\'eodory conditions (see $(C_5)$), $\c\in C(I)$, $\c\ge0$, $\theta\in(0,1)$ and $\eta:I\to I$ is a measurable function such that for a fixed $[a,b]\subset (0,1)$ satisfies $\eta(I)\subset[a,b]$. Note that $u\circ\eta$ is in $L^\infty(I)$.
\par
We could consider more complex BCs or non-linearities, but for the sake of simplicity and insight we will keep it this way.
Observe that the BVP~\eqref{eqpull} is equivalent to
\begin{equation*}
 u(t)= \int_0^1k(t,s)\left[f(s,u(s))+\c(s)u(\eta(s))\right]ds+\theta \max_{s\in[a,b]} u(s),\end{equation*}
where
\begin{displaymath}k(t,s):=\begin{cases} s(1-t), & 0\le s\le t\le 1,\\
t(1-s), & 0\le t\le s\le 1.\end{cases}\end{displaymath}
 Observe that $k$ is non-negative. Take $\Phi(s)=\sup_{t\in I} k(t,s)=s(1-s)$. By direct calculation we obtain
 \begin{displaymath}
 \til\Phi(s):=\inf_{t\in [a,b]}k(t,s)=\begin{cases} s(1-b), & 0\le s \le\frac{a}{1-(b-a)},\\a(1-s), & \frac{a}{1-(b-a)}\le s \le 1.\end{cases}
 \end{displaymath}
Thus, $\inf_{s\in I}\til\Phi(s)/\Phi(s)=\min\{a,1-b\}$, so we  take $c\le \min\{a,1-b\}$. We will look for solutions in the cone
\[K:=\{u\in C(I)\ :\ \min_{t\in[a,b]}u(t)\ge c\|u\|\}.\]
Observe that, for $u\in K$,
\begin{align*}f(t,u(t))+\c(t)u(\eta(t))\le & f(t,u(t))+\c(t)\max_{s\in[a,b]} u(s),\ t\in I,\\
f(t,u(t))+\c(t)\min_{t\in[a,b]}u(t)\le & f(t,u(t))+\c(t)u(\eta(t)),\ t\in[a,b].\end{align*}
Hence, take
\begin{align*}g\equiv & 1;\ f_i=f,\ m_i=n_i=1, i=1,2;\\ \a_{11}[u]= & \b_{11}[u]= \min_{t\in[a,b]}u(t),\\ \a_{21}[u]= & \b_{21}[u]= \max_{s\in[a,b]} u(s),\\ 
\delta_{11}(t)= & \delta_{21}(t)=\theta,\ \c_{11}(t)=\c_{21}(t)=\c(t).
\end{align*}
With these definitions we obtain
\begin{align*}
\varphi_1[u]= & \big(\min_{t\in[a,b]}u(t),\min_{t\in[a,b]}u(t)\big),\\
\varphi_2[u]= & (\max_{t\in[a,b]} u(t),\max_{t\in[a,b]} u(t)),\\
\psi_1[u] = & \psi_2[u]=  \left(\int_0^1|k(t,s)|\gamma(s)ds,\theta \right),\\
M_1= & \begin{pmatrix} m_1  & \theta \\ m_1 & \theta \end{pmatrix},\ r(M_1)=m_1+\theta,\\
M_2= & \begin{pmatrix}m_2 & \theta \\ m_2 & \theta \end{pmatrix},\ r(M_2)=m_2+\theta,\\
\cK_{\varphi_{11}}(s) & =\cK_{\varphi_{12}}(s)=\min\{a(1-s),s(1-b)\},\\
\cK_{\varphi_{21}}(s) & =\cK_{\varphi_{22}}(s)=\begin{cases}s(1-a), & 0\le s\le a,\\ s(1-s), & a\le s\le b,\\ b(1-s), & b\le s\le 1.\end{cases}
\end{align*}
where \[m_1=\min\limits_{t\in[a,b]}\int_0^1|k(t,s)|\gamma(s)ds\text{\quad and\quad} m_2=\max_{t\in[a,b]}\int_0^1|k(t,s)|\gamma(s)ds.\]
Observe that, with these definitions, conditions $(C_1)$--$(C_7)$, $(C_9)$ and $(C_{10})$ are satisfied. Assume also that $r(M_1)<1/\min\{a,1-b\}$ and $r(M_2)<1$. Then we have that $(C_8)$ is also satisfied.\par

If we rewrite the condition $(I_\rho^1)$ in terms of the choices we have made, we get
\[(\Id-M_2)^{-1}=\frac{1}{1-m_2-\theta}\begin{pmatrix}1-\theta & \theta \\ m_2 & 1-m_2 \end{pmatrix},\]
\[\int_0^1\cK_{\varphi_2}(s)g(s)ds=\left(-\frac{a^3}{6}+\frac{b^3}{6}-\frac{b^2}{2}+\frac{b}{2}\right)(1,1),\]
\[(\Id-M_2)^{-1}\int_0^1\cK_{\varphi_2}(s)g(s)ds=\frac{-\frac{a^3}{6}+\frac{b^3}{6}-\frac{b^2}{2}+\frac{b}{2}}{1-m_2-\theta}\left(1,1\right),\]
\[\sigma(t)=\frac{1}{2}t(1-t),\]
and condition $(I_\rho^1)$ becomes
\begin{displaymath}f^{-\rho,\rho}\sup_{t\in I}\(\frac{-\frac{a^3}{6}+\frac{b^3}{6}-\frac{b^2}{2}+\frac{b}{2}}{1-m_2-\theta}\left[\int_0^1|k(t,s)|\gamma(s)ds+\theta\right]+\frac{1}{2}t(1-t)\)<1.\end{displaymath}
Of course, a sufficient condition in order for $(I_\rho^1)$ to be satisfied, which is easier to check, is
\begin{displaymath}f^{-\rho,\rho}\(\frac{-\frac{a^3}{6}+\frac{b^3}{6}-\frac{b^2}{2}+\frac{b}{2}}{1-m_2-\theta}\left[\int_0^1s(1-s)\gamma(s)ds+\theta\right]+\frac{1}{8}\)<1.\end{displaymath}
If we rewrite the condition $(I_\rho^0)$ in terms of the choices we have made, we get

\[(\Id-c_1M_1)^{-1}=\frac{1}{1-c_1(m_1+\theta)}\begin{pmatrix}1-c_1\theta & c_1\theta \\ c_1m_1 & 1-c_1m_1 \end{pmatrix},\]
\[\int_0^1\cK_{\varphi_1}(s)g(s)ds=\left(\frac{a-a b}{2 a-2 b+2},\frac{a-a b}{2 a-2 b+2}\right),\]
\[(\Id-c_1 M_1)^{-1}\int_0^1\cK_{\varphi_1}(s)g(s)ds=\frac{1}{1-c_1(m_1+\theta)}\left(\frac{a-a b}{2 a-2 b+2},\frac{a-a b}{2 a-2 b+2}\right)\]
\[\int_a^bk(t,s)g(s)ds=
 \frac{1}{2} \left(a^2 (t-1)-t ((b-2) b+t)\right),\  a\le t\le b,\]

\[\inf_{t\in[a,b]}\int_a^bk(t,s)g(s)ds=\begin{cases}
 \frac{1}{2} a (a-b) (a+b-2), & a+b\leq 1, \\
 \frac{1}{2} (b-1) (a-b) (a+b), & \text{otherwise,} \end{cases}\]
 and condition $(I_\rho^0)$ becomes
\begin{multline}
  f_{1,\rho,\rho/c}\cdot \inf_{t\in [a,b]}\(\frac{1}{1-c_1(m_1+\theta)}\frac{a-a b}{2 a-2 b+2}\left[\int_0^1|k(t,s)|\gamma(s)ds+\theta\right]\right.\\\left.\phantom{\int} +\left(a^2 (t-1)-t ((b-2) b+t)\right)\)>1.
\end{multline}
A sufficient condition in order for $(I_\rho^0)$ to be satisfied is
 \begin{multline}
  f_{1,\rho,\rho/c}\(\frac{1}{1-c_1(m_1+\theta)}\frac{a-a b}{2 a-2b+2}\left[\int_0^1\min\{a(1-s),s(1-b)\}\gamma(s)ds +\theta\right] \right.\\ \left.+\inf_{t\in[a,b]}\left(a^2 (t-1)-t ((b-2) b+t)\right)\)>1.
\end{multline}

\begin{exa}
Let us now consider a particular case. Take $f(t,u)=tu^2$, $\gamma(t)=t(1-t)+\frac{1}{4}$, $\theta=1/2$ in the BVP~\eqref{eqpull}. Fix $\rho_1=1$, $\rho_2=28$, $a=1/4$, $b=3/4$. With these data, we have $c=1/4$, $f_2^{-\rho_1,\rho_1}=\rho_1=1$, $f_{1,\rho_2,\rho_2/c}=\rho_2/4=7$.\par
\[m_1=\frac{43}{1024}\approxeq 0.0419922,\ m_2=\frac{11}{192}\approxeq 0.0572917.\]
 Condition $(I_{\rho_1}^1)$ is
\begin{displaymath}f_2^{-\rho_1,\rho_1}\sup_{t\in I}\(\frac{31}{85}\left[\frac{1}{24} (t-1) t (2 (t-1) t-5)+\frac{1}{2}\right]+\frac{1}{2}t(1-t)\)<1.\end{displaymath}
where
\[\sup_{t\in I}\(\frac{31}{85}\left[\frac{1}{24} (t-1) t (2 (t-1) t-5)+\frac{1}{2}\right]+\frac{1}{2}t(1-t)\)=\frac{5357}{16320}\approxeq 0.328248\]
so condition $(I_{\rho_1}^1)$ is satisfied and condition $(I_{\rho_2}^0)$ becomes
 \[
  f_{1,\rho_2,\rho_2/c}\cdot \inf_{t\in [a,b]}\(\frac{256}{3541}\left[\frac{1}{24} (t-1) t (2 (t-1) t-5)+\frac{1}{2}\right]+\left(\frac{1}{16} (t-1)-t (t-\frac{15}{16})\right)\)>1,
 \]
 where
\[\inf_{t\in [a,b]}\(\frac{256}{3541}\left[\frac{1}{24} (t-1) t (2 (t-1) t-5)+\frac{1}{2}\right]+\left(\frac{1}{16} (t-1)-t (t-\frac{15}{16})\right)\)=\frac{4651}{28328}\approxeq 0.164184\]
so condition $(I_{\rho_2}^0)$ is satisfied.
\par Therefore $(S_2)$ in Theorem \ref{thmcasesS} is satisfied and the BVP~\eqref{eqpull} has at least a solution which is positive in $[1/4,3/4]$.
\end{exa}
We now apply Theorem \ref{thmindeig} to the BVP
\begin{equation}\label{eqpull2}-u''(t)+u(t)=f(t,u(t))+\theta \, u(\eta(t)),\ u(0)=u(1),u'(0)=u'(1),\end{equation}
where $\theta\in(0,1/2]$,  $f$ satisfies the $L^\infty$-Carath\'eodory conditions and $\eta:I\to I$ is a measurable function such that for a fixed $[a,b]\subset (0,1)$ satisfies $\eta(I)\subset[a,b]$. We rewrite sufficient conditions according to Remarks \ref{remp1}--\ref{remp2}, for the points $(1)-(3)$ to be satisfied. Firstly, problem \eqref{eqpull2} is equivalent to
\begin{equation*}
 u(t)= \int_0^1k(t,s)\left[f(s,u(s))+\theta u(\eta(s))\right]ds,\end{equation*}
 where
\begin{displaymath}k(t,s)=\begin{cases}
 -\frac{e^{s-t+1}+e^{t-s}}{2-2 e}, & 0\le s\le t\le1, \\
-\frac{e^{s-t}+e^{-s+t+1}}{2-2 e}, &  0< t<s\le1.
\end{cases}\end{displaymath}

In this case, we have  that 
\begin{align*}
\varphi_1[u]= & (\min_{t\in[a,b]}u(t),0),\\
\varphi_2[u]= & (\max_{t\in[a,b]} u(t),0),\\
\psi_1[u] = & \psi_2[u]=  \left(\theta,0 \right).\end{align*}
Let us bound $\|L_2\|$ from above, that is
\begin{displaymath}L_2u(t)=\int_0^1|k(t,s)|u(s)ds\le \int_0^1|k(t,s)|ds\|u\|,\end{displaymath}
obtaining
\begin{displaymath}\|L_2\|\le\sup_{t\in[0,1]}\int_0^1|k(t,s)|ds=1.\end{displaymath}

In this case, $H_2u(t)= \theta\,\max_{s\in[a,b]} u(s)$. Note that $\|H_2\|^* \le \theta <1$ and $H_2(K\cap P)=[0,+\infty)\subset C(I)$ is a cone and therefore closed for the sum, which means, by Remark \ref{rangerem}, that $(\Id-H_2)^{-1}$ maps $K\cap P$ to itself. Furthermore, we have that, for $\theta\le 1 /2$, \[(\Id-H_2)^{-1}u(t)=u(t)+\frac{\theta}{1-\theta}\max_{s\in[a,b]} u(s),\ t\in[0,1],\]
which satisfies $(\Id-H_2)^{-1}u\le(\Id-H_2)^{-1}v$, $(\Id-H_2)^{-1}(\l u)\le\l(\Id-H_2)^{-1}u$ for every $u\le v$, $u,v\in P\cap K$, $\l\in\bR^+$.
Also we have $\|(\Id-H_2)^{-1}\|\le1/(1-\theta)$. On the other hand, we have
\begin{displaymath}
\inf_{t\in[a,b]}\int_a^b k(t,s)ds=\frac{e^{a-b+1}-e^{b-a}+1-e}{2-2 e}.
\end{displaymath}
With these values, we have
\begin{itemize}
\item[$(1)$] $0\le f_2^0<1-\theta$,
\item[$(2)$] $0\le \dfrac{2-2 e}{e^{a-b+1}-e^{b-a}+1-e}<f_{1,0}\le\infty,$
\item[$(3)$] $0\le \dfrac{2-2 e}{e^{a-b+1}-e^{b-a}+1-e}<f_{1,\infty}\le\infty.$
\end{itemize}
\begin{exa} Consider again $f(t,u)=tu^2$,  $a=1/4$, $b=3/4$; this time in BVP~\eqref{eqpull2}. We have that $f^0_2=f_{1,0}=0$ and $f_2^\infty=f_{1,\infty}=+\infty$. Hence, the conditions $(1)$ and $(3)$ in Theorem~\ref{thmindeig} are satisfied and therefore, by Lemma~\ref{lemind}, the BVP~\eqref{eqpull2} has at least a nontrivial solution.
\end{exa}

\section*{Acknowledgements}
A. Cabada and F. A. F. Tojo were partially supported by
Ministerio de Educaci\'on y Ciencia, Spain, and FEDER, Project
MTM2013-43014-P.  
F. A. F. Tojo was partially supported by Xunta de Galicia (Spain), project EM2014/032; FPU scholarship, Ministerio de Educaci\'on, Cultura y Deporte, Spain and a Fundaci\'on Barrie Scholarship. 
G. Infante was partially supported by G.N.A.M.P.A. - INdAM (Italy).
This paper was mostly written during a visit of F. A. F. Tojo to the Dipartimento di Matematica e Informatica of the Universit\`a della Calabria.  F. A. F. Tojo is grateful to the people of the aforementioned Dipartimento for their kind and warm hospitality.
The authors wish to acknowledge their gratitude to Dr Filomena Cianciaruso and Professor Paolamaria Pietramala, who helped improve this paper with their fruitful suggestions and comments.

\end{document}